\pdfoutput=1
\documentclass[russian,english]{article}
\usepackage[T1]{fontenc}
\usepackage[cp1251]{inputenc}
\usepackage{textcomp}
\usepackage{ifsym}
\usepackage{amsmath}
\usepackage{amssymb}
\usepackage{amsthm}

\usepackage{babel}
\usepackage{graphicx}
\author{Alexandra Skripchenko\footnote{This work is supported in part by RFBR (grant no. 10-01-91056-{\selectlanguage{russian}НЦНИ\_а})} \\ \selectlanguage {english} Moscow State University}
\title{Symmetric interval identification systems of order three}

\begin{document}
\maketitle
\begin{abstract}
In the present paper we study symmetric interval identification systems of order three. We prove that the Rauzy induction preserves symmetry: for any symmetric interval identification system of order 3 after finitely many iterations of the Rauzy induction we always obtain a symmetric system. We also provide an example of symmetric interval identification system of thin type. 
\end{abstract}

The notion of interval identification system was introduced by I. A. Dynnikov
and B. Wiest in \cite{2} and studied then by I. A. Dynnikov in \cite{1}. It is a generalization of interval exchange transformations and interval translation mappings. 

\theoremstyle{definition}
\newtheorem{definition}{Definition}
\begin{definition}
An \emph {oriented interval identification system} is an object that consists of:
\begin{enumerate}
\item An interval $\left[A,B\right]$ (we call this
interval \emph{the support interval}); 
\item A natural number $n$ (we call this number \emph{the order} of the system);
\item A collection of $n$ unordered pairs $\left\{ \left[a_{i},b_{i}\right],\left[c_{i},d_{i}\right]\right\} $
of subintervals of $\left[A,B\right]$ in each of which the intervals
have equal lengths: $b_{i}-a_{i}=d_{i}-c_{i}$.
\end{enumerate}
\end{definition}
For every pair of intervals $\left\{ \left[a_{i},b_{i}\right],\left[c_{i},d_{i}\right]\right\} $
from an interval identification system we consider the orientation preserving affine isometry between them and we will say that $x$ is identified to $y$ (and write $x\leftrightarrow_{i}y$)
if $x$ is mapped to $y$ or $y$ is mapped to $x$ under this isometry.
So we write $x$$\leftrightarrow_{i}$$y$ if there exists $t\in\left[0,1\right]$
such that $\left\{ x,y\right\} =\left\{ a_{i}+t\left(b_{i}-a_{i}\right),c_{i}+t\left(d_{i}-c_{i}\right)\right\} $. 
A more general object, interval identification system, in which some pairs of intervals are identified by orientation reversing maps is not considered in this paper. All interval identification systems in this paper are assumed to be oriented. 

Objects similar to interval identification systems have appeared in the theory of {$\mathbb R$}-trees (sometimes without giving them specific name) as an instrument for describing the leaf space of a band complex (see \cite{3} and \cite{7} for details). 

\begin{definition}
An interval identification system is called \emph{balanced}, if
$A=\min_{i}(a_{i}), B=\max_{i}(b_{i})$ and ${\textstyle \sum_{i=1}^{n}\left(b_{i}-a_{i}\right)}=B-A$.
\end{definition}

\begin{definition}
An interval identification system is called \emph{symmetric} if $a_{i}-A$$=$ $B$$-d_{i}$ for each of
$i=1,\ldots,n$. 
\end{definition}

The motivation for studying balanced interval identification systems of order three comes from the three-dimensional topology and is connected with Novikov's problem \cite{5} of asymptotic behavior of plane sections of three-periodic surfaces. More precisely, interval identification systems are used in \cite{1} to construct 3-periodic surfaces in the three-space whose intersections with plane of a fixed direction has a chaotic behavior. Novikov's problem originates from the conductivity theory for monocrystals immersed in a magnetic field. Symmetric interval identification systems correspond to Fermi surfaces that are invariant under a central symmetry, which is always the case in all physically meaningful examples.

With each interval identification system
$$S=\left(\left[A,B\right];\left[a_{1},b_{1}\right]\leftrightarrow\left[c_{1},d_{1}\right];\left[a_{2},b_{2}\right]\leftrightarrow\left[c_{2},d_{2}\right];\left[a_{3},b_{3}\right]\leftrightarrow\left[c_{3},d_{3}\right]\right)$$
we associate a graph $\Gamma\left(S\right)$ whose vertices are all points of the support interval, and two vertices of the graph are connected by an edge if and only if these two
points are identified by our system in the sense that is described
above. The system $S$ determines an equivalence relation
$\sim$ on the support interval $\left[A,B\right]$: the points lying
on the same connected component of the graph $\Gamma\left(S\right)$ are
said to be \emph{equivalent}. The set of points equivalent in this sense
to $x$ is called \emph{the orbit} of $x$ in S. The connected component of our graph that contains a vertex $x\in [A,B]$ will be denoted by $\Gamma_{x}(S)$.

We are interested in the properties of orbits of an interval identification
systems such as finiteness and being everywhere dense. For studying them a special Euclid type algorithm is used. The analog of this process that appears in the theory of interval exchange transformation is called the Rauzy induction. This process can also be considered as a particular
case of the Rips machine algorithm for band complexes in the theory of
{$\mathbb R$}-trees (see \cite{4} and \cite{7} for details). The main idea is that from any interval identification system one constructs a sequense of interval identifiacation systems equivalent in a certain sense to the original one (see the precise definition below) but with a smaller support. Combinatorial properties of this sequence are responsible for "ergodic" properties of the original interval identification system. 

\begin{definition}
Two interval identification systems $S_{1}$ and $S_{2}$ with supports $[A_{1},B_{1}]$ and $[A_{2},B_{2}]$, respectively, are called \emph{equivalent}, if there is a real number $t \in {\mathbb R}$ and an interval $[A,B] \subset [A_{1},B_{1}] \cap  [A_{2}+t,B_{2}+t]$ such that
\begin {enumerate}
\item every orbit of each of the systems $S_{1}$ and $S_{2}+t$ contains a point lying in $[A,B]$
\item for each point $x\in [A,B]$ the graphs $\Gamma_{x}(S_{1})$ and $\Gamma_{x}(S_{2}+t)$ are homotopy equivalent through mappings that are identical on $[A,B]$ and such that the full preimage of each vertex contains only finitely many vertices of the other graph.
\end {enumerate}
\end{definition}

\noindent It is easy to see that it is an equivalence relation. 

\emph{The Rauzy induction} for an interval identification system is a recursive application of admissible transmissions followed by reductions as described below. 

\begin{definition}
Let $$S=\left(\left[A,B\right];\left[a_{1},b_{1}\right]\leftrightarrow\left[c_{1},d_{1}\right];\left[a_{2},b_{2}\right]\leftrightarrow\left[c_{2},d_{2}\right];\left[a_{3},b_{3}\right]\leftrightarrow\left[c_{3},d_{3}\right]\right)
$$
be an interval identification system and let one of the subintervals,
$\left[c_{1},d_{1}\right]$, say, be contained in another one $\left[c_{2},d_{2}\right]$, say. Let $S^{'}$ be the interval
identification system obtained from $S$ by replacing the pair $\left[a_{1},b_{1}\right]\leftrightarrow\left[c_{1},d_{1}\right]$
with $\left[a_{1},b_{1}\right]\leftrightarrow\left[c_{1}^{'},d_{1}^{'}\right]$
where $\left[c_{1}^{'},d_{1}^{'}\right]$=$\left[c_{1},d_{1}\right]-c_{2}$+$a_{2}$$\subset$$\left[a_{2},b_{2}\right]$.
We say that $S^{'}$is obtained from $S$ by a \emph{transmission} (of $[c_{1},d_{1}]$ along $[a_{2},b_{2}]\leftrightarrow [c_{2},d_{2}]$).

If, in addition, we have $c_{2}=A$, then this operation is called
an \emph{admissible transmission on the left, }and if $d_{2}=B$,
an \emph{admissible transmission on the right}. 
\end {definition}

\begin{definition}
Let 
$$
S=\left(\left[A,B\right];\left[a_{1},b_{1}\right]\leftrightarrow\left[c_{1},d_{1}\right];\left[a_{2},b_{2}\right]\leftrightarrow\left[c_{2},d_{2}\right];\left[a_{3},b_{3}\right]\leftrightarrow\left[c_{3},d_{3}\right]\right)
$$
be an interval identification system and let $d_{1}=B$. We call all
endpoints of our subintervals \emph{critical points}. Assume that
the point $B$ is not covered by any interval from $S$ except $d_{1}$
and that the interior of the interval $\left[c_{1},d_{1}\right]$ contains
a critical point. Let $u$ the rughtmost such point. Then the interval $[u,B]$ is covered by only one interval from our system. Replacing the pair $\left[a_{1},b_{1}\right]\leftrightarrow\left[c_{1},d_{1}\right]$
with $\left[a_{1},b_{1}-d_{1}+u\right]\leftrightarrow$$\left[c_{1},u\right]$
in $S$ with simultaneous cutting off the part $[u,B]$ from the support interval will be called a\emph{ reduction on the right} (of the pair $\left[a_{1},b_{1}\right]\leftrightarrow\left[c_{1},d_{1}\right]$). A reduction
on the left is defined in the symmetric way. 
\end{definition}

An example of an iteration (transmission on the right plus reduction on the right) of the Rauzy induction to a symmetric interval identification system is shown in Figure \ref{void}. 
\begin{figure}
\includegraphics[width=8cm,height=8cm]{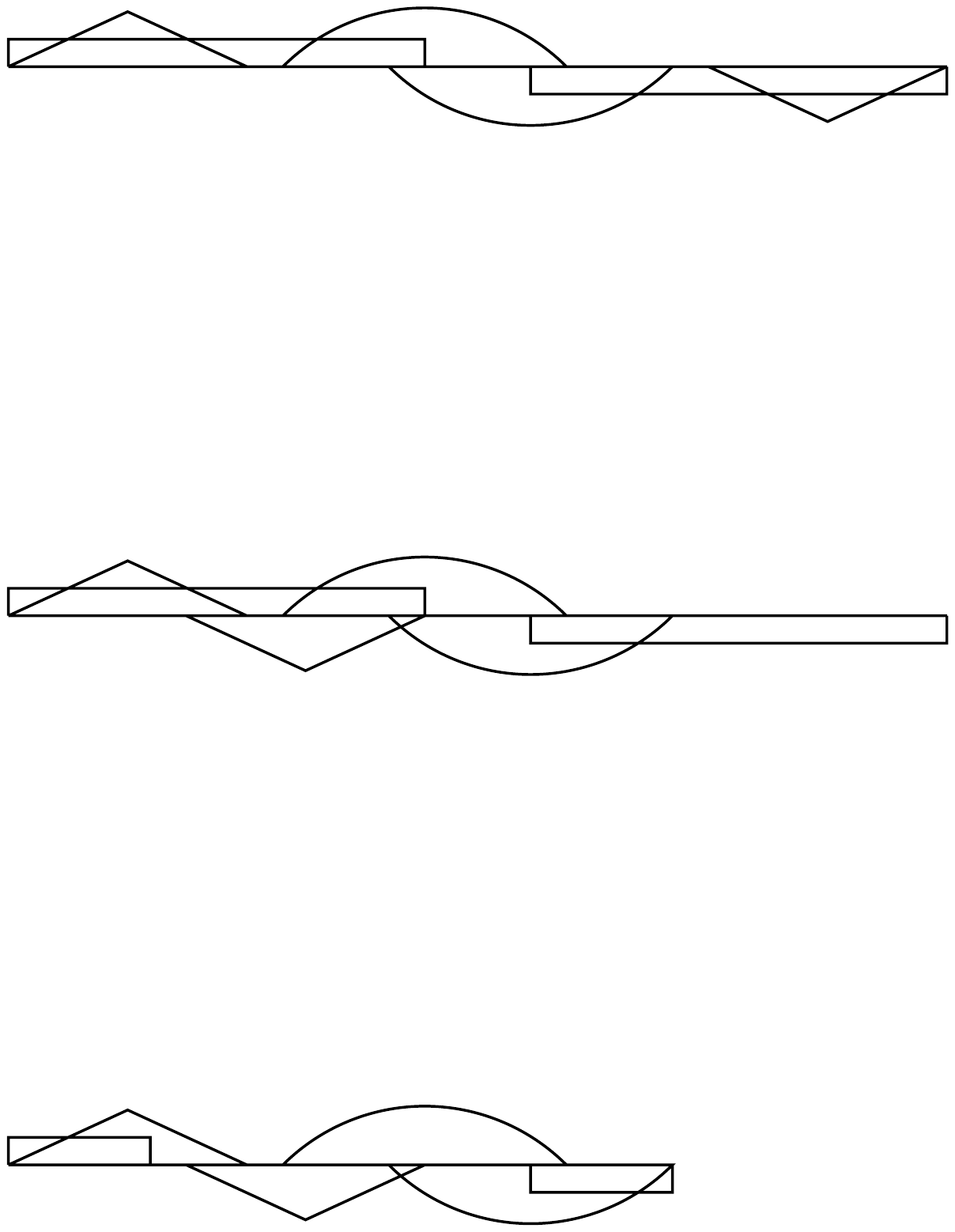}
\put(-225,155){$a_{1}=a_{2}$}
\put(-180,149){$b_{2}$}
\put(-177,154){\vector(0,2){8}}
\put(-152,149){$b_{1}$}
\put(-151,154){\vector(0,2){8}}
\put(-139,171){{$c_{1}$}}
\put(-136,170){\vector(0,-3){8}}
\put(-111,171){$c_{2}$}
\put(-110,170){\vector(0,-3){8}}
\put(-93,163){$d_{1}=d_{2}$}
\put(-140,145){\vector(0,-10){20}}
\put(-135,135){transmission of $[c_{2},d_{2}]$ on the right}
\put(-140,85){\vector(0,-10){20}}
\put(-135,75){reduction of $[c_{1},d_{1}]$ on the right}
\caption{An iteration of the Rauzy induction}
\label{void}
\end{figure}

We say that an interval identification system has a \emph{hole}
if there are some points in the support interval that are not covered by
a interval from S. This means in particular that our system has points
with finite orbits. 

The Rauzy induction stops once a system with a hole obrained. One can show that transmissions and reductions turn
an interval identification system into an equivalent one. We are interested in describing symmetric interval identification systems that are not equivalent to a system with a hole. By applying one reduction on either
side to such a system we always obtain a system with $a_{1}=a_{2}=A$ and $d_{1}=d_{2}=B$, up to renumeration of the intervals. We call systems satisfying this condition \emph{special}. If, in addition, we fix
the interval $\left[A,B\right]$, then a special symmetric interval
identification system is left three degrees of freedom, and has the form (we assume $A$=0 for simplicity):
\begin{equation}
\begin{split}
\label{100}
S=([0,a+b+c];& [0,a]\leftrightarrow[b+c,a+b+c],\\ 
					& [0,b]\leftrightarrow[a+c,a+b+c], \\
					& [u,u+c]\leftrightarrow[a+b-u,a+b+c-u])
\end{split}
\end{equation}
with $a,b,c,u>0$, $a+b+c=B-A$.

We are interested in the most generic case of symmetric interval identification system in the sense that no integral linear relation holds for the parameters $a,b,c,u$ except those that must hold by definition.

We define a \emph{generalized iteration} of the Rauzy induction by analogy with a step of the fast version of Euclid's algorithm, which involves the division with remainder instead of subtraction of the smaller number from the larger. It may happen that only one of the three pairs of intervals is subject to reduction in several consecutive steps of the Rauzy induction (and the intervals from the second and the third pair are involved only in transmissions). In this case we consider the result of such a sequence of the Rauzy induction iterations as the result of applying of one generalized iteration. An example is shown in Figure 6.

For a symmetric interval identification
system of order three one step of a one-side
Rauzy induction (for example, admissible transmission on the right
followed by reduction on the right) doesn't preserve the symmetry. However, we have the following.
\theoremstyle{theorem}
\newtheorem{theorem}{Theorem}
\begin{theorem} For any special symmetric balanced interval identification
system of order three without hole $S$, at most three generalized iteration of a one side Rauzy induction is needed to obtain a new special symmetric balanced system or a system with a hole.
\end{theorem}
\begin{proof} By changing notation if necessary we may assume that the system $S$ has a form~(\ref{100}) with $a>b>0$, $c>0$, $a+b-u>u>0$. Indeed,
if the inequality $a>b$ does not hold we simply exchange $a$ and $b$, and if we
have $a+b-u<u$, we replace $u$ by $a+b-u$, which does not change the system. We
call the intervals in $S$ of lengths $a$, $b$ and $c$ $a$-intervals, $b$-intervals and
$c$-intervals, respectively, and the corresponding pairs of intervals the $a$-pair, the $b$-pairs and the $c$-pairs respectively. We also call the point relative to which the intervals in the $a$-pair ($b$-pair, $c$-pair, respectively) are symmetric the $a$-midpoint ($b$-midpoint, $c$-midpoint, respectively). The interval identification system is symmetric iff these three points coincide. Note that the process stops if we obtain the system with a hole; so we assume that during all iterations described below a hole will not appear. We divide all generic interval
identification systems of this form (without a hole) into eight groups listed below
according to the relative position of the intervals.

\begin{description}
\item [Case 1:] $u<b-c$;
\item [Case 2:] $b-c<u<b$, $a>b+c$;
\item [Case 3:] $b-c<u<b$, $a<b+c$;
\item [Case 4:] $b<u<u+c$, $a<u+c$;
\item [Case 5:] $b<u<b+c$, $u+c<a+b-u$;
\item [Case 6:] $b<u<b+c$, $a+b-u<u+c<a$;
\item [Case 7:] $b+c<u$, $u+c<a+b-u$;
\item [Case 8:] $b+c<u$, $u+c<a$, $u+c>a+b-u$.
\end{description}
\smallskip
In Case 1 the abscense of hole implies $a>b+c$, and the order of the critical points in the support interval is:
$$u<u+c<b<b+c<a<a+c<a+b-u<a+b+c-u.$$
One can show that in this case two ordinary (not generalized) iterations of the Rauzy
induction result in a symmetric system. The first iteration consists of a transmission of $b$-interval along the $a$-pair 
and a reduction of $a$-pair on the right. The second iteration consists of a transmission of $c$-interval along the $a$ -pair and a reduction on the right of the rest of $a$-pairs. See an example in Figure \ref{2}, where $a$-intervals are represented by rectangular, $b$-intervals are represented by triangular and $c$-intervals by circular arcs.

\begin{figure}
\includegraphics[width=8cm,height=8cm]{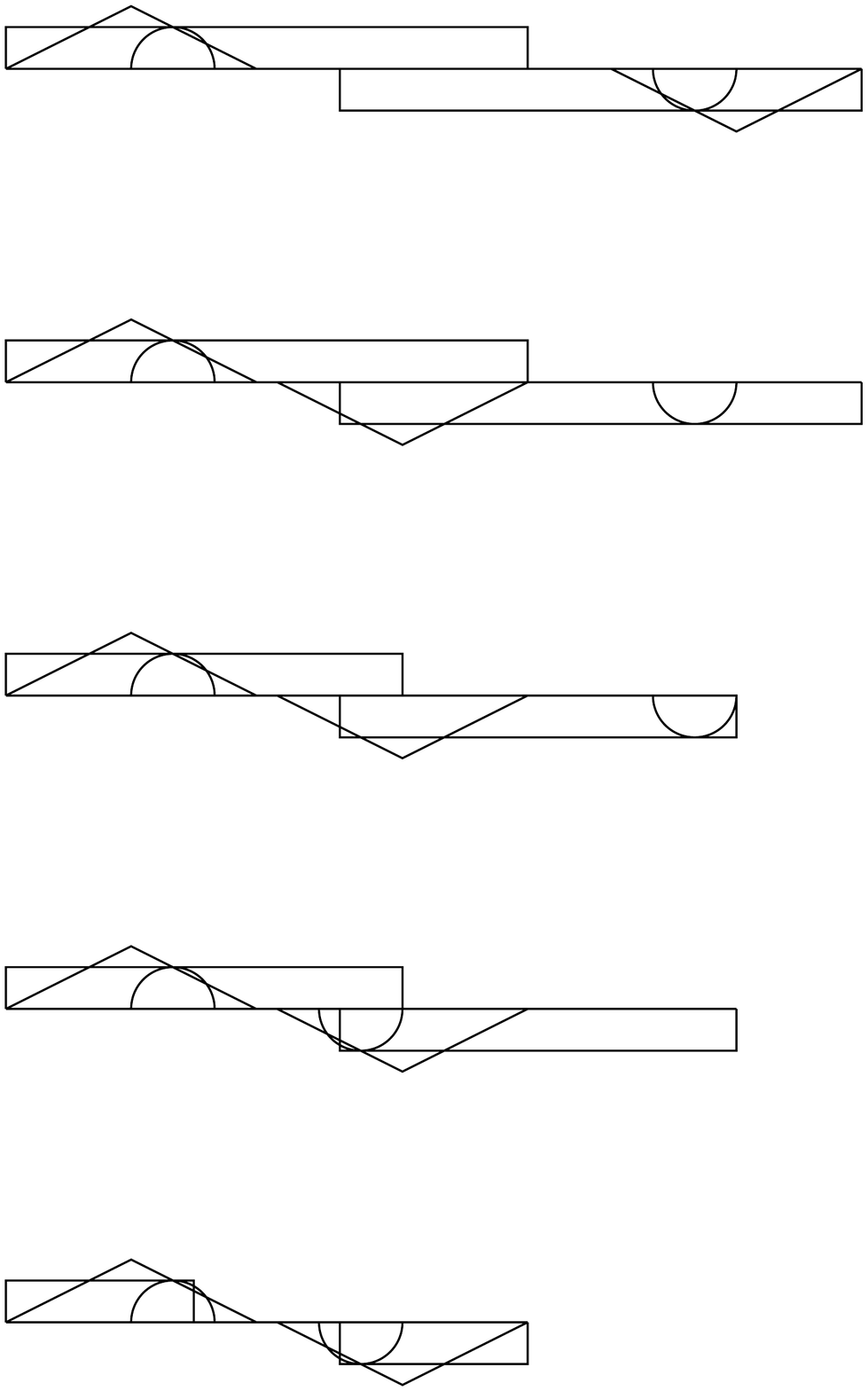}
\put(-170,196){$0$}
\put(-74,203){$a$}
\put(-28,203){$a+b+c$}
\put(-126,211){$b$}
\put(-150,197){$u$}
\put(-144,191){$u+c$}
\put(-132,196){\vector(0,10){8}}
\put(-125,210){\vector(0,-10){8}}
\put(-90,192){\vector(0,-10){20}}
\put(-90,150){\vector(0,-10){20}}
\put(-90,108){\vector(0,-10){20}}
\put(-90,66){\vector(0,-10){20}}
\caption{The Rauzy induction in Case 1}
\label{2}
\end{figure}

We obtain a new symmetric system (maybe with a hole) of the form~(\ref{100}) with new parameters $a',b',c',u'$,
where 

$\left(\begin{array}{c}
a'\\
b'\\
c'\\
u'\end{array}\right)$$=A$$\left(\begin{array}{c}
a\\
b\\
c\\
u\end{array}\right)$, and $A$ is one of the following matrices:

$\left(\begin{array}{cccc}
1 & -1 & -1 & 0\\
0 & 1 & 0 & 0\\
0 & 0 & 1 & 0\\
0 & 0 & 0 & 1\end{array}\right)$, $\left(\begin{array}{cccc}
1 & -1 & -1 & 0\\
0 & 1 & 0 & 0\\
0 & 0 & 1 & 0\\
1 & 0 & -1 & -1\end{array}\right)$,

$\left(\begin{array}{cccc}
0 & 1 & 0 & 0\\
1 & -1 & -1 & 0\\
0 & 0 & 1 & 0\\
0 & 0 & 0 & 1\end{array}\right)$, $\left(\begin{array}{cccc}
0 & 1 & 0 & 0\\
1 & -1 & -1 & 0\\
0 & 0 & 1 & 0\\
1 & 0 & -1 & -1\end{array}\right)$.

\noindent We choose the matrix for which the following inequalities hold: 
$$a'>b',  
a'+b'-u'>u'\eqno (2)$$
\medskip
\noindent In Case 2 the order of the critical points in the support interval is
$$u<b<u+c<b+c<a<a+b-u<a+c<a+b+c-u.$$
The scheme of the Rauzy induction is exactly the same as in the previous case. 

\medskip
\noindent
In Case 3 the abscense of a hole implies $a<2u+c-b$, and the order of the critical points in the support interval is:
$$u<b<a<a+b-u<u+c<b+c<a+c<a+b+c-u.$$
The scheme of Rauzy induction is as follows: the first iteration consists of a transmission of a $b$-interval and a reduction on the right of the $a$-pair; as a result the $a$-midpoint shifts by $\frac{u}{2}$ to the left, the $b$-midpoint shifts on $\frac{b+c}{2}$ to the left, the $c$-midpoint does not move. Denote ${[\frac{c+u-a}{a+b-2u}]}+1$ (where $[x]$ is the integer part of $x$) by $k$. The next $k-1$ iterations consists of a transmission of an $a$-interval on the right and a reduction of the $c$-intervals (one generalized iteration). During these iterations the $b$-midpoint does not move, while the $a$-midpoint and the $c$-midpoint shift by the same distance. The last generalized iteration consists of a transmission of a $c$-interval and a reduction of the $a$-pair. After that the rightmost point of the right interval in the $a$-pair coincides with the rightmost point of the right interval in the $b$-pair, therefore, the $a$-midpoint coincides with the $b$-midpoint. The last shift of the $a$-midpoint is $$\frac{(a+b+c-u-k(a+b-2u))}{2} -\frac{a}{2}$$ and the last shift of the $c$-midpoint is $$\frac{(a+b+c-k(a+b-2u))}{2}-\frac{a}{2},$$ so the total shifts of the $a$-midpoint and the $b$-midpoint in the mentioned operations are equal. Therefore, all three midpoints coincide after this procedure and we obtain a new symmetric system (maybe with a hole) of the form~(\ref{100}) with new parameters $a',b',c',u'$,
where

\selectlanguage{english}%
$\left(\begin{array}{c}
a'\\
b'\\
c'\\
u'\end{array}\right)$$=B$$\left(\begin{array}{c}
a\\
b\\
c\\
u\end{array}\right)$, and $B$ is one of the following matrices:

$\left(\begin{array}{cccc}
1+k & -1+k & -1 & -2k\\
0 & 1 & 0 & 0\\
-k & -k & 1 & 2k\\
0 & 0 & 0 & 1\end{array}\right)$, $\left(\begin{array}{cccc}
1+k & -1+k & -1 & -2k\\
0 & 1 & 0 & 0\\
-k & -k & 1 & 2k\\
1+k & k & -1 & -1-2k\end{array}\right)$,

$\left(\begin{array}{cccc}
0 & 1 & 0 & 0\\
1+k & -1+k & -1 & -2k\\
-k & -k & 1 & 2k\\
0 & 0 & 0 & 1\end{array}\right)$, $\left(\begin{array}{cccc}
0 & 1 & 0 & 0\\
1+k & -1+k & -1 & -2k\\
-k & -k & 1 & 2k\\
1+k & k & -1 & -1-2k\end{array}\right)$. 

\noindent
As in the previous cases we choose the matrix for which the inequalities (2) hold.
So the system becomes symmetric after three generalized iterations. See an example in Figure \ref{4}. 
\begin{figure}
\includegraphics[width=8cm,height=8cm]{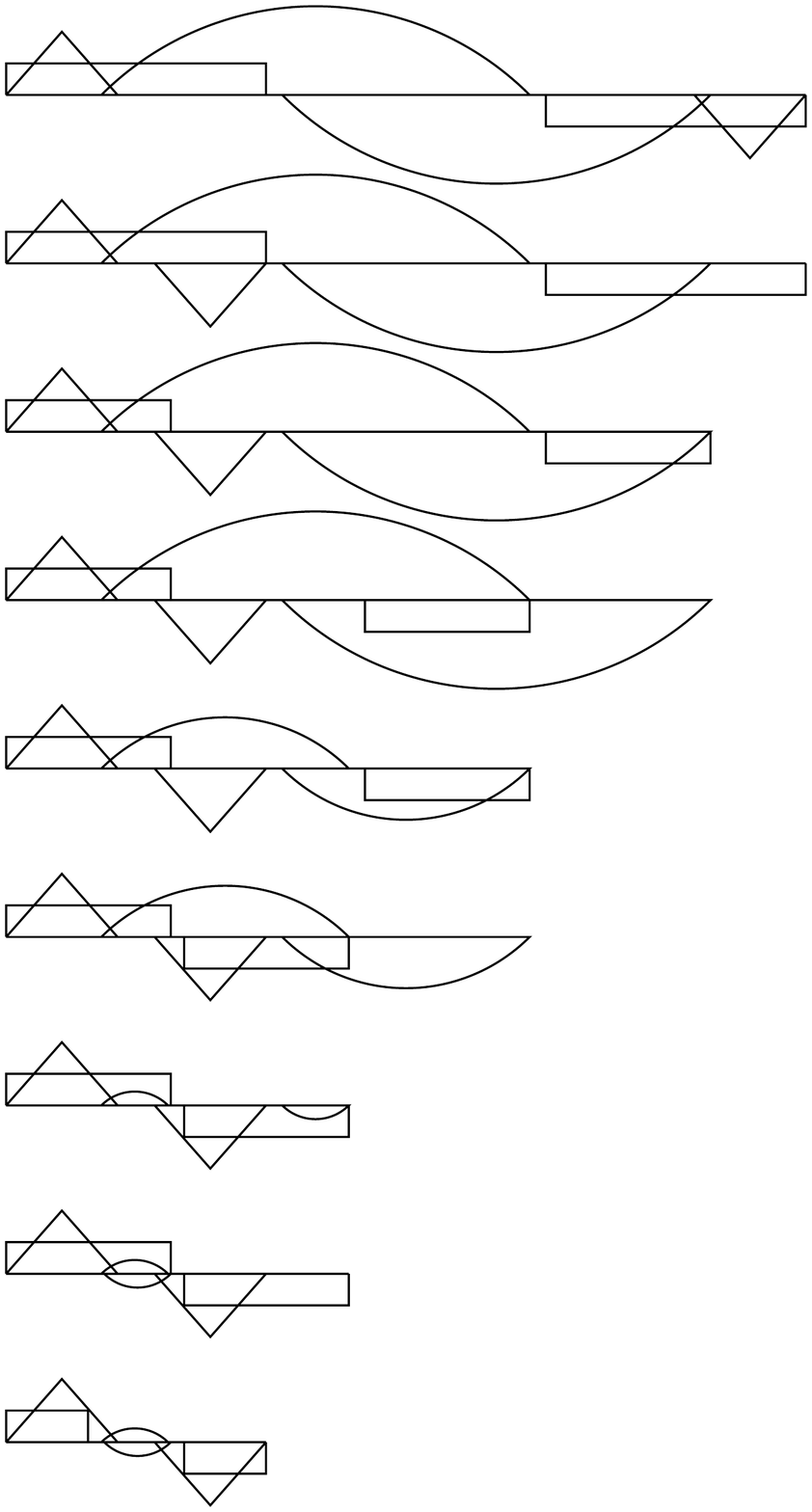}
\put(-160,188){\vector(0,-10){9}}
\put(-160,166){\vector(0,-10){9}}
\put(-160,144){\vector(0,-10){9}}
\put(-160,122){\vector(0,-10){9}}
\put(-160,100){\vector(0,-10){9}}
\put(-160,78){\vector(0,-10){9}}
\put(-160,56){\vector(0,-10){9}}
\put(-160,34){\vector(0,-10){9}}
\put(-139,195){$a$}
\put(-175,189){$u$}
\put(-171,203){$b$}
\put(-170,202){\vector(0,-10){8}}
\put(-97,204){$u+c$}
\put(-91,204){\vector(-1,-2){5}}
\put(-191,187){$0$}
\put(-70,195){$a+b+c$}
\caption{The Rauzy induction in Case 3}
\label{4}
\end{figure}

\medskip
\noindent
In Case 4 the order of the critical points in the support interval is:
$$b<u<a+b-u<a<u+c<a+b+c-u<a+c.$$
The scheme of the Rauzy induction is exactly the same as in the previous case. 

\medskip
\noindent
In Case 5 the abscense of a hole means, in particular, that $a>b+c$ and the order of critical points in the support interval is: 
$$b<u<b+c<u+c<a+b-u<a<a+b+c-u<a+c.$$
The Rauzy induction starts with a transmission of a $b$-interval along the $a$-pair and a reduction of the $a$-pair because $a+b+c-u$ becomes the rightmost point of the remaining support interval. The next iteration consists of a transmission of $c$-interval along the $a$-pair and then a reduction of the $a$-pair. 
An example is shown in Figure \ref{5}.
\selectlanguage{english} 
\begin{figure}
\includegraphics[width=10cm,height=10cm]{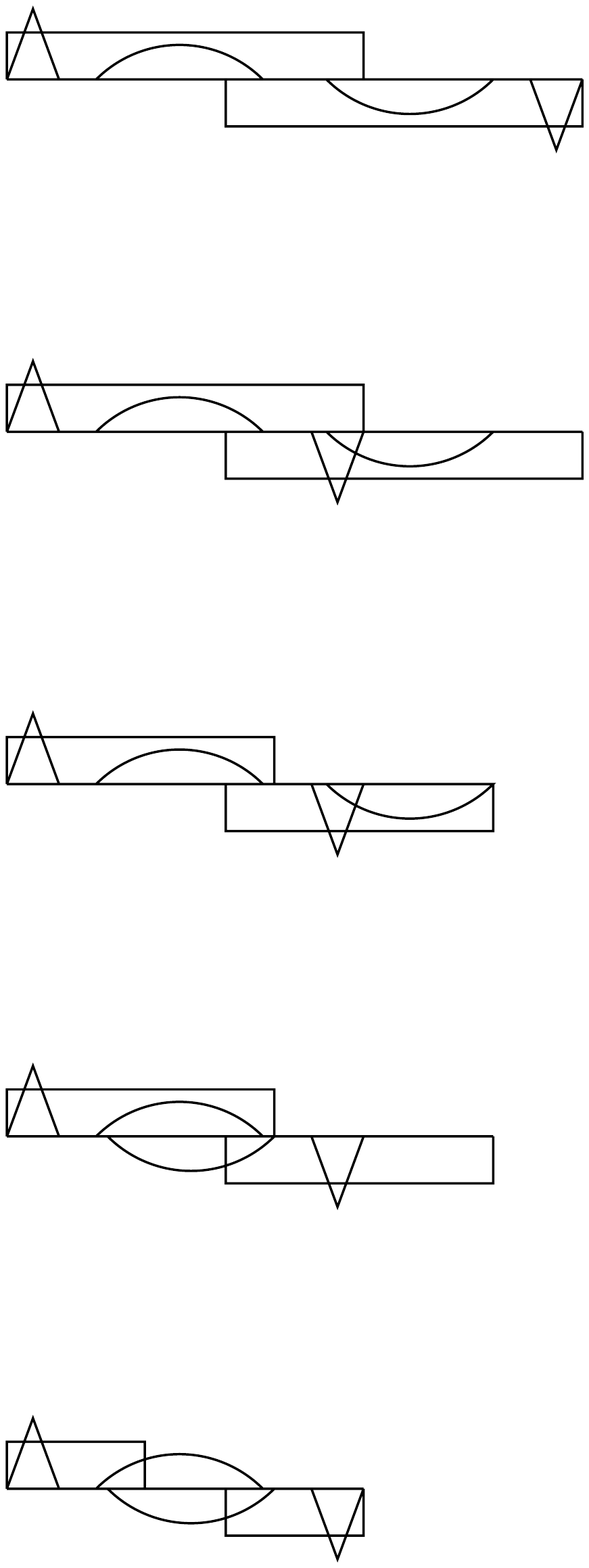}
\put(-175,243){\vector(0,-10){20}}
\put(-175,184){\vector(0,-10){20}}
\put(-175,130){\vector(0,-10){20}}
\put(-175,80){\vector(0,-10){20}}
\put(-162,255){$a$}
\put(-230,245){$b$}
\put(-220,245){$u$}
\put(-197,265){$u+c$}
\put(-186,264){\vector(0,-10){11}}
\put(-245,245){$0$}
\put(-120,255){$a+b+c$}
\caption{The Rauzy induction in Case 5}
\label{5}
\end{figure}

\noindent So, two ordinary iterations result in a new symmetric system (maybe with a hole) of the form~(\ref{100}) with new parameters $a',b',c',u'$, where  

$\left(\begin{array}{c}
a'\\
b'\\
c'\\
u'\end{array}\right)$$=A$$\left(\begin{array}{c}
a\\
b\\
c\\
u\end{array}\right)$, and $A$ is one of the following matrices:

$\left(\begin{array}{cccc}
1 & -1 & -1 & 0\\
0 & 1 & 0 & 0\\
0 & 0 & 1 & 0\\
0 & 0 & 0 & 1\end{array}\right)$, $\left(\begin{array}{cccc}
1 & -1 & -1 & 0\\
0 & 1 & 0 & 0\\
0 & 0 & 1 & 0\\
1 & 0 & -1 & -1\end{array}\right)$,

$\left(\begin{array}{cccc}
0 & 1 & 0 & 0\\
1 & -1 & -1 & 0\\
0 & 0 & 1 & 0\\
0 & 0 & 0 & 1\end{array}\right)$, $\left(\begin{array}{cccc}
0 & 1 & 0 & 0\\
1 & -1 & -1 & 0\\
0 & 0 & 1 & 0\\
1 & 0 & -1 & -1\end{array}\right)$.

\noindent We choose the matrix for which the inequalities (2) hold.

\medskip
\noindent
In Case 6 the order of the critical points in the support interval is:
$$b<u<b+c<a+b-u<u+c<a<a+b+c-u<a+c.$$
The scheme of the Rauzy induction is exactly the same as in the previous case.

\medskip
\noindent
In Case 7 the order of the critical points in the support interval is:
$$b<b+c<u<u+c<a+b-u<a+b+c-u<a<a+c.$$
In this case the first iteration of the Rauzy induction consists of a transmission of a $b$-interval and a reduction of the $a$-pair. Three midpoints move to the following positions: the $a$-midpoint and the $b$-midpoint to $\frac{a}{2}$ and the $c$-midpoint to $\frac{a+b+c}{2}$. There are two possibilities: $a-b<b$ or $a-b-c>b$ since in the case of $2b<a<2b+c$ we obtain a system with a hole.

If we have $a-b<b$ and, therefore, $a-b-c<b$, then each of $a$-intervals is located inside the corresponding $b$-interval. We will refer to the left $c$-interval as the $c_{1}$-interval and the right one the $c_{2}$-interval. The next $3n$ iterations, where $n=[\frac{b}{a-b}]$, consist of transmissions of $c$-intervals or $a$-interval to the right and reductions of the $b$-intervals. More precisely, there are $n$ blocks of iterations, in each of which the order of iterations is as follows:
\begin{itemize}
\item a transmission of the $a$-interval plus a reduction of the $b$-pair; 
\item a transmission of the $c_{1}$-interval plus a reduction of the $b$-pair; 
\item a transmission of the $c_{2}$-interval plus a reduction of the $b$-pair. 
\end{itemize}
Three midpoints shift by the same distance. As a result we obtain an interval identification system for which the rightmost point of the $c_{1}$-interval is $u+c-n(a-b)$,  the rightmost point of the $c_{2}$-interval is $a+b+c-u-n(a-b)$ and the rightmost point of the second interval in $a$-pair is $b-m(a-b)$, where $m$=$n-1$. One can check that now $b$-intervals are located inside $a$-intervals. The next two iterations consist of subsequent transmissions of $c$-intervals and reductions of $a$-intervals. The next iteration consists of a transmission of a $b$-interval and a reduction of the $a$-pair. Three midpoints get to the following positions: the $a$-midpoint coincide with the $b$-midpoint (because the leftmost ends of left intervals in corresponding pairs coincide and so do the rightmost points of the right intervals), the $a$-midpoint shifts in total by $b+c$ and the $c$-midpoint shifts by one half of $$-2n(a-b)+b+c+((a-u)-(a+b+c-u-n(a-b)))+(u-b-(u+c-n(a-b)))=2(b+c).$$  An example is shown in Figure \ref{711}.
\selectlanguage{english}
\begin{figure}
\includegraphics[width=12cm,height=12cm]{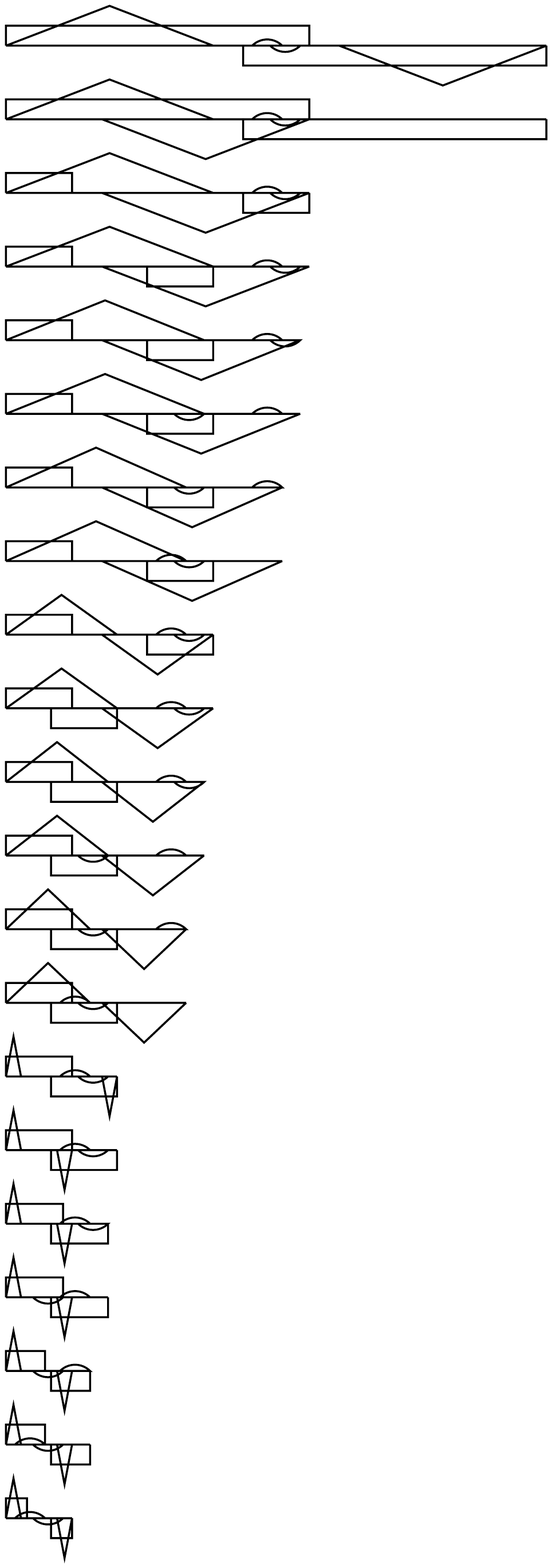}
\put(-300,315){\vector(0,-10){8}}
\put(-300,300){\vector(0,-10){8}}
\put(-300,285){\vector(0,-10){8}}
\put(-300,270){\vector(0,-10){8}}
\put(-300,255){\vector(0,-10){8}}
\put(-300,240){\vector(0,-10){8}}
\put(-300,225){\vector(0,-10){8}}
\put(-300,210){\vector(0,-10){8}}
\put(-300,195){\vector(0,-10){8}}
\put(-300,180){\vector(0,-10){8}}
\put(-300,162){\vector(0,-10){8}}
\put(-300,147){\vector(0,-10){8}}
\put(-300,130){\vector(0,-10){8}}
\put(-300,115){\vector(0,-10){8}}
\put(-300,100){\vector(0,-10){8}}
\put(-300,85){\vector(0,-10){8}}
\put(-300,70){\vector(0,-10){8}}
\put(-300,55){\vector(0,-10){8}}
\put(-300,40){\vector(0,-10){8}}
\put(-300,25){\vector(0,-10){8}}
\put(-235,320){$a$}
\put(-270,310){$b$}
\put(-253,326){\vector(1,-2){6}}
\put(-255,327){$u$}
\put(-327,310){$0$}
\put(-162,310){$a+b+c$}
\caption{The Rauzy induction in Case 7: the first example}
\label{711}
\end{figure}

One can check that these iterations result in interval identification system (maybe with a hole) of the form~(\ref{100}) with new parameters $a',b',c',u'$,
where

$\left(\begin{array}{c}
a'\\
b'\\
c'\\
u'\end{array}\right)$$=C$$\left(\begin{array}{c}
a\\
b\\
c\\
u\end{array}\right)$, and $C$ is the matrix from the list:

$\left(\begin{array}{cccc}
1+n & -n-2 & -2 & 0\\
-n & 1+n & 0 & 0\\
0 & 0 & 1 & 0\\
0 & -1 & 0 & 1\end{array}\right)$, 

$\left(\begin{array}{cccc}
-n & 1+n& 0 & 0\\
1+n & -n-2 & 0 & 0\\
0 & 0 & 1 & 0\\
0 & -1 & 0 & 1\end{array}\right)$,
with $\mathit{n=[\frac{b}{a-b}]}$,
for which the inequalities (2) hold.

If $a-b-c>b$ holds, $b$-intervals are located inside $a$-intervals. The next $x=[\frac{u}{b+c}]$ iterations consist of transmissions of the $b$-interval and reductions of the $a$-pair.  Next $2(y-x)$ iterations, where $y=[\frac{a-u-c}{b+c}]$, consist of transmissions of a $b$-interval and transmissions of the $c_{2}$-interval, each followed by a reduction of the $a$-pair, in the alternating order. After all these iterations the rightmost point of the right $b$-interval is $a-y(b+c)$, the rightmost points of $c$-intervals are still $u+c$ and $a+b+c-u-(y-x)(b+c)$. In each transmission of a $b$-interval the $b$-midpoint shifts by $\frac{b+c}{2}$ to the left; in each transmission of the $c_{2}$-interval the $c$-midpoint shifts by the same distance. The next $3x-1$ (if $a-y(b+c)>a+b+c-u-(y-x)(b+c)$) or $3x-2$ (in the other case) iterations consist of transmissions of the $c_{1}$-interval, transmissions of the $c_{2}$-interval, and transmissions of a $b$-interval, each followed by a reduction of the $a$-pair, in the order shown below:
$$
\underbrace{c_{1},c_{2}, b, c_{1}, c_{2},b,\ldots,c_{1},c_{2}}_{\mbox{$3x-1$}}
$$
or
$$
\underbrace{c_{1}, b, c_{2}, c_{1}, b, c_{2},\ldots,c_{1}}_{\mbox{$3x-2$}}.
$$
In each transmission of the $b$-interval the $b$-midpoint shifts by $$\frac{b+c}{2}$$ to the left; in each transmission of the $c_{2}$-interval followed by transmission of $c_{1}$-interval the $c$-midpoint shifts by the $b+c$. So, after all described iterations the rightmost point of the second interval in the $a$-pair coincides with the rightmost point of the second interval in the $b$-pair, therefore, the $a$-midpoint coincides with the $b$-midpoint. The $b$-midpoint shifts in total by $$\frac{(x+y-1)(b+c)}{2}$$ and the $c$-midpoint shifts by $$\frac{(y-x-1+2x)(b+c)}{2}.$$ An example of this situation is shown in Figure 6.
\begin{figure}
\includegraphics[width=12cm,height=17cm]{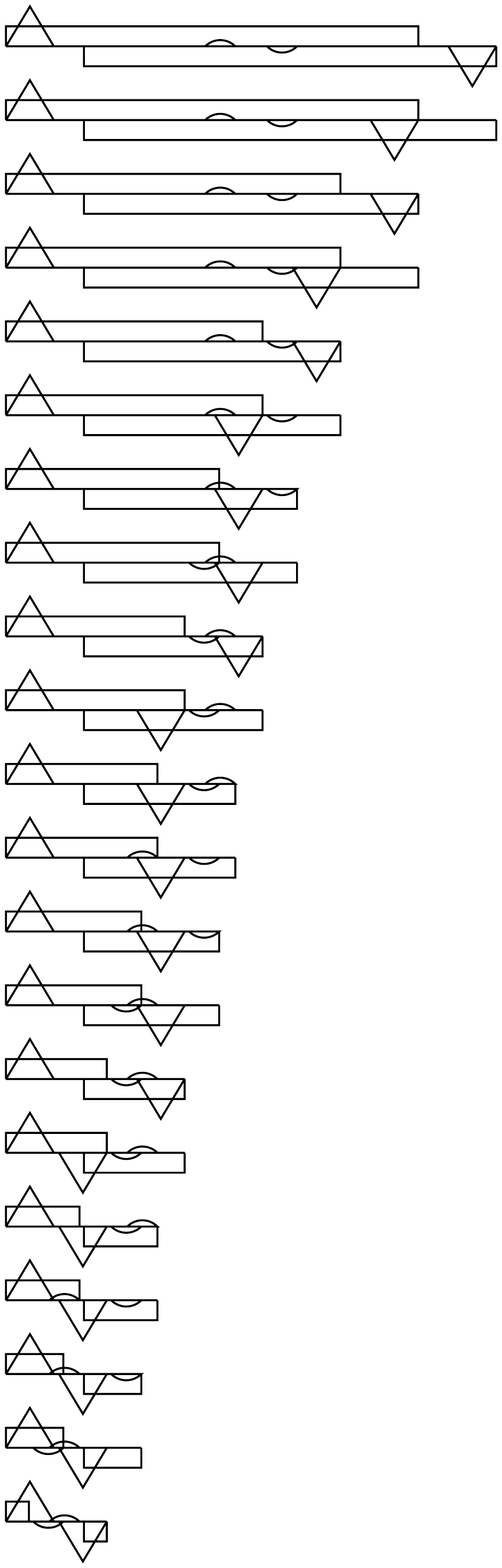} 
\put(-300,440){\vector(0,-10){7}}
\put(-320,441){$0$}
\put(-310,441){$b$}
\put(-265,454){$u$}
\put(-264,454){\vector(0,-1){9}}
\put(-255,454){\vector(0,-1){9}}
\put(-255,454){$u+c$}
\put(-208,441){$a$}
\label{7}
\put(-300,420){\vector(0,-10){7}}
\put(-300,395){\vector(0,-10){7}}
\put(-300,374){\vector(0,-10){7}}
\put(-300,355){\vector(0,-10){7}}
\put(-300,335){\vector(0,-10){7}}
\put(-300,313){\vector(0,-10){7}}
\put(-300,290){\vector(0,-10){7}}
\put(-300,269){\vector(0,-10){7}}
\put(-300,249){\vector(0,-10){7}}
\put(-300,227){\vector(0,-10){7}}
\put(-300,207){\vector(0,-10){7}}
\put(-300,187){\vector(0,-10){7}}
\put(-300,164){\vector(0,-10){7}}
\put(-300,144){\vector(0,-10){7}}
\put(-300,122){\vector(0,-10){7}}
\put(-300,98){\vector(0,-10){7}}
\put(-300,77){\vector(0,-10){7}}
\put(-300,55){\vector(0,-10){7}}
\put(-300,33){\vector(0,-10){7}}
\caption{The Rauzy induction in Case 7: the second example}
\end{figure}

Therefore all three midpoints coincide after this procedure and the process results either in a system with a hole or in the system of the form~(\ref{100}) with new parameters $a',b',c',u'$,
where $\left(\begin{array}{c}
a'\\
b'\\
c'\\
u'\end{array}\right)$$=C$$\left(\begin{array}{c}
a\\
b\\
c\\
u\end{array}\right)$, and $C$ is the matrix from the list:

$\left(\begin{array}{cccc}
1 & -(x+y)-1 & -(x+y)-1 & 0\\
0 & 1 & 0 & 0\\
0 & 0 & 1 & 0\\
1 & -y & -y-1 & -1\end{array}\right)$, 

$\left(\begin{array}{cccc}
0 & 1 & 0 & 0\\
1 & -(x+y)-1 & -(x+y)-1 & 0\\
0 & 0 & 1 & 0\\
1 & -y & -y-1 & -1\end{array}\right)$,
with 
$\mathit{x=[\frac{u}{b+c}],y=[\frac{a-u-c}{b+c}]}$, for which the inequalities (2) hold.

\medskip
\noindent
In Case 8 the order of critical points in the support interval is:
$$b<b+c<u<a+b-u<u+c<a<a+b+c-u<a+c.$$
The scheme of the Rauzy induction is exactly the same as in the previous case. 
\end{proof}

Now we construct a symmetric interval identification system of thin
type. By the latter we mean an interval identification system for which an equivalent system may have arbitrarily small support (in \cite{8} such interval translation mappings are called ITM of infinite type). As described in \cite{1} such systems can be used to construct 3-periodic central symmetric surfaces in the three space whose intersections with plane of fixed direction has a chaotic behavior. The construction is motivated by Novikov's problem on conductivity of normal metals \cite{5}. Thin case in the theory of {$\mathbb R$}-trees was discovered by
G. Levitt in \cite{4}. A concrete example of a  translation map of thin case was provided by M. Boshernitzan and I. Kornfeld in \cite{8} and H. Bruin and S. Troubetzkoy proved in \cite{9} that thin interval translation maps form a set of zero measure. A generic example of order $3$ interval identification system of thin type is given by Dynnikov in \cite{1}. A construction equivalent to this example was described in different terms in \cite{6}. A concrete example of thin band complexes is given by M. Bestvina and M. Feighn in \cite{3}. None of these examples was symmetric.

Denote by $M$ the following matrix: 

$\left(\begin{array}{cccc}
3 & 1 & -1 & -4\\
-1 & 2 & 0 & 0\\
-2 & -2 & 1 & 4\\
3 & 2 & -1 & -5\end{array}\right)$

\selectlanguage{english}%
\inputencoding{latin9}\noindent It is easy to see that this matrix has exactly
one real positive eigenvalue $\lambda<1$. Its approximate value
is $\lambda\approx0.254$. 

\newtheorem{proposition}{Proposition}
\begin{proposition}
Let $\left(a,b,c,u\right)$ be an eigenvector
of the matrix $M$ with the eigenvalue $\lambda$ and positive coordinates.
Then the corresponding symmetric interval identification system 
\begin{equation*}
\begin{split}
S=(\left[0,a+b+c\right];& \left[0,a\right]\leftrightarrow\left[b+c,a+b+c\right],\\ 
					& \left[0,b\right]\leftrightarrow\left[a+c,a+b+c\right], \\
					& \left[u,u+c\right]\leftrightarrow\left[a+b-u,a+b+c-u\right])
\end{split}
\end{equation*}
is of thin type. Approximate values of $\left(a,b,c,u\right)$ normalized by $a+b+c=1$ are equal to $\left(0.444,0.254,0.302,0.292\right)$. 
\end{proposition}
\begin{proof}
It is easy to see that after 6 iterations of the
right-side Rauzy induction the resulting interval identification
system is a scaled down version of the original one multiplied by $\lambda$: one can check that given values of parameters determine a system
that corresponds to Case 4 from the previous proposition with $k=2$;
after an iteration of the Rauzy induction that corresponds to this case we obtain a symmetric
system related to Case 2 from the theorem. After the next two ordinary iterations we return
to Case 4. The Rauzy induction for our example is shown in Figure \ref{thin}.
\begin{figure}
\includegraphics[width=10cm,height=10cm]{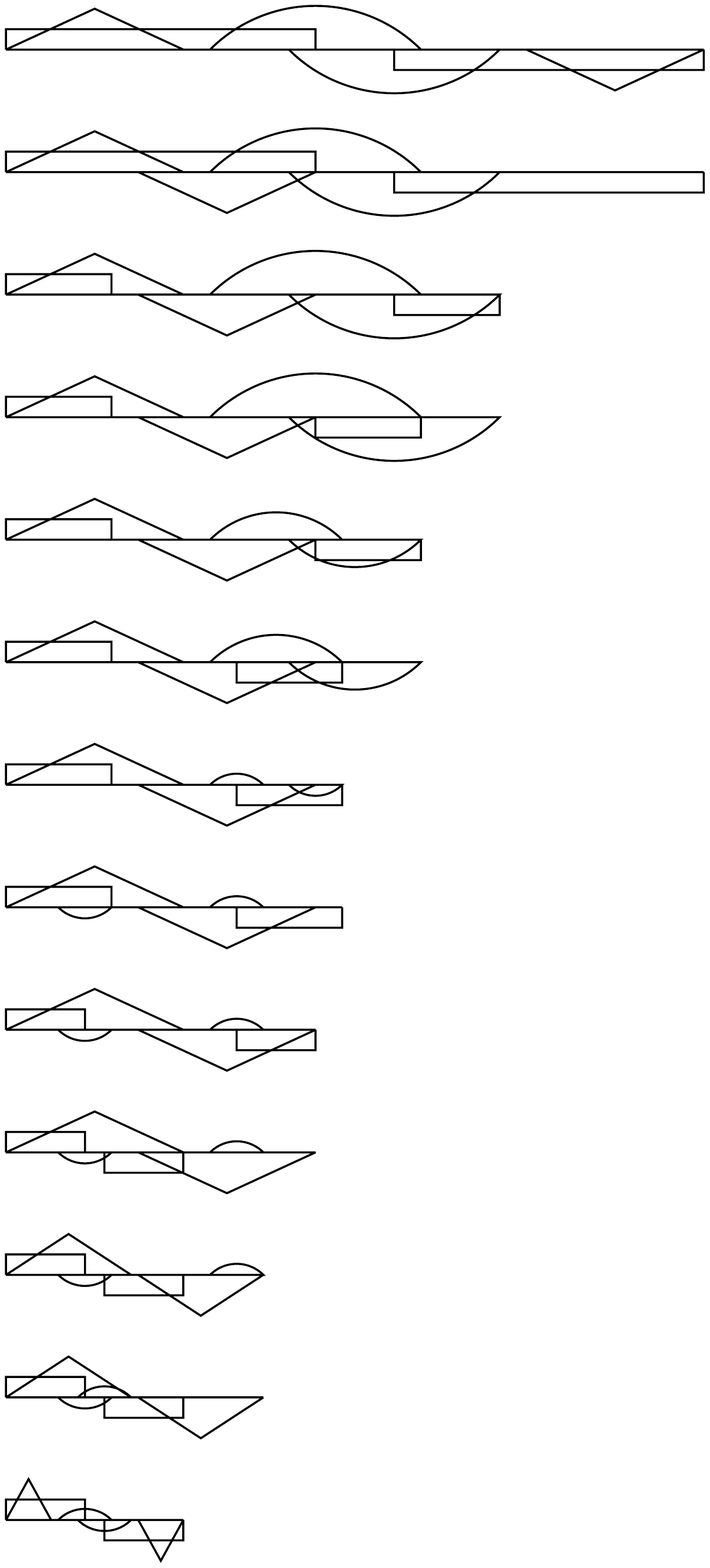}
\put(-210,260){\vector(0,-10){9}}
\put(-210,240){\vector(0,-10){9}}
\put(-210,220){\vector(0,-10){9}}
\put(-210,200){\vector(0,-10){9}}
\put(-210,178){\vector(0,-10){9}}
\put(-210,156){\vector(0,-10){9}}
\put(-210,136){\vector(0,-10){9}}
\put(-210,116){\vector(0,-10){9}}
\put(-210,96){\vector(0,-10){9}}
\put(-210,76){\vector(0,-10){9}}
\put(-210,56){\vector(0,-10){9}}
\put(-210,36){\vector(0,-10){9}}
\put(-201,256){$b$}
\put(-190,256){$u$}
\put(-160,275){$a$}
\put(-159,275){\vector(-1,-2){6}}
\put(-135,275){\vector(-1,-2){6}}
\put(-142,275){$u+c$}
\put(-239,255){$0$}
\put(-87,265){$a+b+c$}
\caption{The Rauzy induction in symmetric thin case example}
\label{thin}
\end{figure}

So, by applying sufficiently many steps of the Rauzy induction we can get an interval identification system with arbitrarily small support.
\end{proof}

\emph{Acknowledgements.} I wish to thank I. Dynnikov for posing the problem and for constant attention to this work and T. Coulbois for useful remarks.

\end{document}